\documentclass[11pt]{amsart}

\newtheorem{theorem}{Theorem}[section]

\newtheorem{corollary}[theorem]{Corollary}

\newtheorem{lemma}[theorem]{Lemma}

\newtheorem{problem}[theorem]{Problem}
\newtheorem{proposition}[theorem]{Proposition}

\theoremstyle{definition}

\usepackage[colorlinks, bookmarks=true]{hyperref}
\usepackage{color,graphicx,shortvrb}

\usepackage{enumerate}
\usepackage{amsmath}
\usepackage{amssymb}

\usepackage[latin 1]{inputenc}

\def\J#1#2#3{ \left\{ #1,#2,#3 \right\} }

\def\11{\textbf{$1$}}

\begin{document}

\numberwithin{equation}{section}

\title[2-local triple homomorphisms]{2-local triple homomorphisms on von Neumann algebras and JBW$^*$-triples}

\author[Burgos]{Maria Burgos}
\email{maria.burgos@uca.es}
\address{Departamento de Matematicas, Facultad de Ciencias Sociales y de la Educacion,  Universidad de Cadiz, 11405, Jerez de la Frontera, Spain.}

\author[Fern\'{a}ndez-Polo]{Francisco J. Fern\'{a}ndez-Polo}
\email{pacopolo@ugr.es}
\address{Departamento de An{\'a}lisis Matem{\'a}tico, Facultad de
Ciencias, Universidad de Granada, 18071 Granada, Spain.}

\author[Garc{\' e}s]{Jorge J. Garc{\' e}s}
\email{jgarces@correo.ugr.es}
\address{Departamento de An{\'a}lisis Matem{\'a}tico, Facultad de
Ciencias, Universidad de Granada, 18071 Granada, Spain.}

\author[Peralta]{Antonio M. Peralta}
\email{aperalta@ugr.es}
\address{Departamento de An{\'a}lisis Matem{\'a}tico, Facultad de
Ciencias, Universidad de Granada, 18071 Granada, Spain.}
\curraddr{Visiting Professor at Department of Mathematics, College of Science, King Saud University, P.O.Box 2455-5, Riyadh-11451, Kingdom of Saudi Arabia.}

\thanks{Authors partially supported by the Spanish Ministry of Science and Innovation,
D.G.I. project no. MTM2011-23843, and Junta de Andaluc\'{\i}a grant FQM375.
The fourth author extends his appreciation to the Deanship of Scientific Research at King Saud University (Saudi Arabia) for funding the work through research group no. RGP-361.}



\date{}
\maketitle

\begin{abstract} We prove that every (not necessarily linear nor continuous) 2-local triple homomorphism from a JBW$^*$-triple into a JB$^*$-triple is linear and a triple homomorphism. Consequently, every 2-local triple homomorphism from a von Neumann algebra (respectively, from a JBW$^*$-algebra) into a C$^*$-algebra (respectively, into a JB$^*$-algebra) is linear and a triple homomorphism.
\end{abstract}

\maketitle
\thispagestyle{empty}

\section{Introduction}\label{sec:intro}

It is known that the Gleason-Kahane-\.{Z}elazko theorem (cf. \cite{Gle,KaZe,Ze68}) admits a reinterpretation affirming that every unital linear local homomorphism from a unital complex Banach algebra $A$ into $\mathbb{C}$ is multiplicative. Formally speaking, the notions of local homomorphisms and local derivations were introduced in 1990, in papers due to Larson and Sourour \cite{LarSou} and Kadison \cite{Kad90}. We recall that given two Banach algebras $A$ and $B,$ a linear mapping $T: A\to B$ (respectively, $T:A\to A$) is said to be a \emph{local homomorphism} (respectively, a \emph{local derivation}) if for every $a$ in $A$ there exists a homomorphism  $\Phi_{a} : A \to B$ (respectively, a derivation $D_a :A\to A$), depending on $a$, satisfying $T(a) = \Phi_a (a)$ (respectively, $T(a) = D_a (a)$). A flourishing research on linear local homomorphisms and derivations was built upon the results of Kadison, Larson and Sourour (compare, for example, \cite{AlbAyuKudayNurj2011, BattMol, BreSemrl93, BreSemrl95, BurFerGarPe2014, BurFerPe2013, CristLocalAutomorph, CristLocalder, Fos2012,
John01, KimKim04, KimKim05}, \cite{Mack}
--\cite{Semrl2010} and \cite{ZhanPanYang}, among the over 100 references on the subject).\smallskip

If in the definition of \emph{local homomorphism}, we relax the assumption concerning linearity with a 2-local behavior, we are led to the notion of (not necessarily linear) \emph{2-local homomorphism}. Let $A$ and $B$ be two C$^*$-algebras, a not necessarily linear nor continuous mapping $T: A\to B$ is said to be a \emph{2-local homomorphism} (respectively, \emph{2-local $^*$-homomorphism}) if for every $a,b\in A$ there exists a bounded (linear) homomorphism (respectively, $^*$-homomorphism) $\Phi_{a,b}: A\to B$, depending on $a$ and $b$, such that $\Phi_{a,b} (a) = T(a)$ and $\Phi_{a,b}(b) = T(b)$ (see \cite{Semrl97}, \cite{BurFerGarPe2014preprint}).\smallskip

In a recent contribution, we establish a generalization of the Kowalski-S{\l}odkowski theorem for 2-local $^*$-homomorphisms on von Neumann algebras, showing that every (not necessarily linear nor continuous) 2-local $^*$-homomorphism from a von Neumann algebra or from a compact C$^*$-algebra into a C$^*$-algebra is linear and a $^*$-homomorphism. In the Jordan setting, it is proved that every 2-local Jordan $^*$-homomorphism from a JBW$^*$-algebra into a JB$^*$-algebra is linear and a Jordan $^*$-homomorphism (cf. \cite{BurFerGarPe2014preprint}).\smallskip

Every C$^*$-algebra $A$ admits a ternary product given by $$\{a,b,c\} := \frac12 (a b^*c + cb^* a) \ (a,b,c\in A).$$ A linear map $\Phi$  between C$^*$-algebras $A$ and $B$ satisfying $\Phi \left(\{a,b,c\}\right) = \{\Phi(a),\Phi(b),\Phi (c) \},$ is called a \emph{triple homomorphism}. A \textit{2-local triple homomorphism} between $A$ and $B$ is a not necessarily linear nor continuous map $T:A \to B$ such that for every $a,b\in A,$ there exists a triple homomorphism $\Phi_{a,b} : A\to B$ with $\Phi_{a,b} (a) = T(a)$ and $\Phi_{a,b}(b) = T(b)$. Motivated by the above commented Kowalski-S{\l}odkowski theorem for von Neumann algebras, it seems natural to consider the following independent problem:

\begin{problem}\label{problem 2-local triple hom C*-algebras}
Is every 2-local triple homomorphism between C$^*$-algebras (automatically) linear?
\end{problem}

It should be noted here that, even in the case of von Neumann algebras, the proofs and arguments given in the study of 2-local $^*$-homomorphisms \cite{BurFerGarPe2014preprint}, are no longer valid when considering Problem \ref{problem 2-local triple hom C*-algebras}, because triple homomorphisms between C$^*$-algebras do not preserve the natural partial order given by the positive cone in a C$^*$-algebra.\smallskip

Problem \ref{problem 2-local triple hom C*-algebras} can be posed in the more general setting of JB$^*$-triples.  Let $E$ and $F$ be two JB$^*$-triples (see subsection \ref{subsect prelimi} for definitions). A linear map $\Phi: E\to F$ which preserves the triple products is called a \emph{triple homomorphism}. A (not necessarily linear nor continuous) mapping $T: E\to F$ is said to be a \emph{2-local triple homomorphism}  if for every $a,b\in E$ there exists a bounded (linear) triple homomorphism $\Phi_{a,b}: E\to F$, depending on $a$ and $b$, such that $\Phi_{a,b} (a) = T(a)$ and
$\Phi_{a,b}(b) = T(b)$. According to these definitions, we consider the following generalization of Problem \ref{problem 2-local triple hom C*-algebras}:

\begin{problem}\label{problem 2-local triple hom JB*-triples}
Is every 2-local triple homomorphism between JB$^*$-triples (automatically) linear?
\end{problem}

In this paper we solve Problems \ref{problem 2-local triple hom C*-algebras} and \ref{problem 2-local triple hom JB*-triples} when the domain is a von Neumann algebra or a JBW$^*$-triple, respectively. Our main result (Theorem \ref{t 2-local triple hom on JBW$^*$ triples}) asserts that every (not necessarily linear nor continuous) 2-local triple homomorphism from a JBW$^*$-triple into a JB$^*$-triple is linear and a triple homomorphism, and consequently, every 2-local triple homomorphism from a von Neumann algebra (respectively, from a JBW$^*$-algebra) into a C$^*$-algebra (respectively, into a JB$^*$-algebra) is linear and a triple homomorphism (cf. Theorem \ref{t 2-local triple hom on JBW$^*$ algebras} and Corollary \ref{c 2-local triple hom on von Neumann algebras}). Our proofs heavily rely on the Bunce-Wright-Mackey-Gleason theorem for JBW$^*$-algebras \cite{BuWri89} and deep geometric arguments and techniques, developed in the setting of JB$^*$-triples by R. Braun, W. Kaup and H. Upmeier \cite{BraKaUp,KaUp}, B. Russo and Y. Friedman \cite{FriRu85}, and G. Horn \cite{Horn}.

\subsection{Preliminaries}\label{subsect prelimi}

A \emph{JB$^*$-triple} is a complex Banach space, $E,$ together with
a continuous triple product $\{ .,.,.\} : E\times E \times E \to E$, $(a,b,c) \mapsto \{a,b,c\}$, which is
conjugate-linear in $b$ and symmetric and bilinear in
$(a,c)$ and satisfies: \begin{enumerate}\item The \emph{Jordan identity}: $$L(a,b) L(x,y) -L(x,y) L(a,b) =
L(L(a,b)x,y) - L(x,L(b,a)y),$$ where $L(a,b)$ denotes the operator given by
$L(a,b) x = \{a,b,x\};$
\item $L(a,a)$ is an hermitian operator with non-negative spectrum;
\item $\|\J aaa\| = \|a\|^3,$
\end{enumerate}
every $a,b,x$ and $y$ in $E$.\smallskip

The notion of JB$^*$-triples was introduced by Kaup in the holomorphic classification of bounded symmetric domains in \cite{Ka83}. One of the many kindness exhibited by the class of JB$^*$-triples is that every C$^*$-algebra (respectively, every JB$^*$-algebra) is a
JB$^*$-triple with respect to $$\{ a,b,c\} := \frac12 ( ab^* c +c b^*
a)$$ (respectively, $\{ a,b,c\} := (a\circ b^*) \circ c + (c\circ b^*) \circ
a - (a\circ c) \circ b^*$).\smallskip

A \emph{JBW$^*$-triple} is a JB$^*$-triple which is also a dual
Banach space (with a unique isometric predual \cite{BarTi}). It is
known that the triple product of a JBW$^*$-triple is separately
weak$^*$ continuous (cf. \cite{BarTi}).\smallskip

We recall that an element $e$ in a JB$^*$-triple $E$ is said to be a \emph{tripotent} if $\{e,e,e\}=e$.
It is known that for each tripotent $e$ in $E$ we have a decomposition (called the \emph{Peirce
decomposition})
$$E= E_{2} (e) \oplus E_{1} (e) \oplus E_0 (e),$$ where for $j=0,1,2,$ $E_j (e)$ is
the $\frac{j}{2}$-eigenspace of $L(e,e)$. The Peirce subspaces $E_j(e)$ satisfy the following multiplication
rules: $$\left\{ {E_{i}(e)},{E_{j} (e)},{E_{k} (e)} \right\}\subseteq E_{i-j+k} (e),$$ if $i-j+k \in \{ 0,1,2\}$ and is zero
otherwise, and
$$\left\{ {E_{2} (e)},{E_{0}(e)},{E} \right\}= \left\{ {E_{0} (e)},{E_{2}(e)},{E} \right\} =0.$$ These multiplication rules are called the \emph{Peirce rules}.
The natural projection $P_{j} (e) : E\to E_{j} (e)$ of $E$ onto $E_j (e)$ is called the \emph{Peirce-$j$ projection}. The Peirce projections are contractive and satisfy $$P_2 (e) = L(e,e)(2 L(e,e) -Id), \ P_1 (e) = 4 L(e,e)(Id-L(e,e)),$$ $$ \ \hbox{ and } P_0 (e) = (Id-L(e,e))
(Id-2 L(e,e)),$$ where $Id$ denotes the identity map on $E$ (compare \cite{FriRu85}).  It is also known
that for each $x_0\in E_0 (e)$ and  $x_2\in E_2 (e)$ we have $\|x_0 + x_2 \| = \max \{ \|x_0\| ,\| x_2\|\}$ (c.f. \cite[Lemma 1.3]{FriRu85}). The tripotent $e$ is called \emph{complete} when $E_0(e)=\{0\}.$ \smallskip

Another interesting property of the Peirce decomposition asserts that $E_{2} (e)$ is a unital JB$^*$-algebra
with unit $e$, product $a \circ_{e} b = \J aeb$ and involution $a^{\sharp_{e}} = \J eae$ (c.f. \cite[Theorem 2.2]{BraKaUp} and
\cite[Theorem 3.7]{KaUp}).\smallskip

Accordingly to the standard terminology, for each element $a$ in a JB$^*$-triple $E$, we denote $a^{[1]} =
a$ and $a^{[2 n +1]} := \J a{a^{[2n-1]}}a$ $(\forall n\in \mathbb{N})$. It follows from the Jordan identity that JB$^*$-triples
are power associative, that is, $\J{a^{[2k-1]}}{a^{[2l-1]}}{a^{[2m-1]}}=a^{[2(k+l+m)-3]}$. In this paper, the symbol  $E_a$
will denote the JB$^*$-subtriple of $E$ generated by $a$. It is known that $E_a$ is JB$^*$-triple isomorphic (and hence isometric) to $C_0 (L)$ for some locally compact Hausdorff space $L\subseteq (0,\|a\|],$ such that $L\cup \{0\}$ is compact and $\|a\| \in L$. It is further known that there exists a triple isomorphism $\Psi$ from $E_a$ onto $C_{0}(L),$ satisfying $\Psi (a) (t) = t$ $(t\in L)$ (compare \cite[Lemma 1.14]{Ka83}). In particular, for each natural $n$, there exists (a unique) element $a^{[{1}/({2n-1})]}$ in $E_a$ satisfying $(a^{[{1}/({2n-1})]})^{[2n-1]} = a.$\smallskip

When $a$ is a norm one element in a JBW$^*$-triple $E$, the sequence $(a^{[{1}/({2n-1})]})$ converges in the weak$^*$ topology of $E$ to a tripotent in $E$, which is denoted by $r(a)$ and is called the \emph{range tripotent} of $a$. The tripotent $r(a)$ is the smallest tripotent $e$ in $E$ satisfying that $a$ is
positive in the JBW$^*$-algebra $E_{2} (e)$ (cf. \cite[Lemma 3.3]{EdRu88}).\smallskip

We refer to \cite{Chu2012} for a recent monograph on JB$^*$-triples and JB$^*$-algebras.\smallskip

Throughout the paper, when $A$ is a C$^*$-algebra or a JB$^*$-algebra, the symbol $A_{sa}$ will
stand for the set of all self-adjoint elements in $A$.

\section{Generalities on 2-local triple homomorphisms}\label{sec:generalities}

We recall that elements $a$ and $b$ in a JB$^*$-triple $E$ are said to be
\emph{orthogonal} (written $a\perp b$) when $L(a,b) =0$. It is known that $a\perp b$ if and only if $\{a,a,b\}=0$, if and only if $\{a,b,b\}=0$ (cf. \cite[Lemma 1.1]{BurFerGarMarPe}). A pair of subsets $M,N\subset E$ are called orthogonal ($M\perp N$) if for every $a\in M$, $b\in N$, we have $a\perp b$.\smallskip

Throughout the paper, given a 2-local triple homomorphism $T$ between JB$^*$-triples $E$ and $F$, for each $a,b\in E$, $\Phi_{a,b}$ will denote a (linear) triple homomorphism satisfying $T(a) = \Phi_{a,b} (a)$ and $T(b) = \Phi_{a,b}(b)$.\smallskip

We begin with some basic properties of 2-local triple homomorphisms.

\begin{lemma}\label{l basic properties}
Let $T: E\to F$ be a (not necessarily linear nor continuous)
2-local triple homomorphism between JB$^*$-triples. The
following statements hold:\begin{enumerate}[$(a)$]
\item $T$ is 1-homogeneous, that is, $T(\lambda a) = \lambda T(a)$
for every $a\in E$, $\lambda\in \mathbb{C}$;
\item $T$ is orthogonality preserving;
\item $\J{T(a)}{T(a)}{T(a)} = T(\J aaa)$, for every $a\in E$. In particular, every linear 2-local triple homomorphism between JB$^*$-triples is a triple homomorphism;
\item $T$ maps tripotents in $E$ to tripotents in $F$;
\item For each $a,b\in E$, $\|T(a)-T(b) \| \leq \|a-b\|$, that is, $T$ is 1-lipschitzian and hence continuous;
\item For each tripotent $e$ in $E$ with $T(e)\neq 0$, we have $T(E_j(e))\subseteq F_j(T(e)),$ for every $j=0,1,2$, $T(E_2(e)+E_1(e))\subseteq
    F_2(T(e))+F_1(T(e))$, and $T(E_0(e)+E_1(e))\subseteq
    F_0(T(e))+F_1(T(e)).$ Furthermore, $T(E_2(e)_{sa})\subseteq F_2(T(e))_{sa};$
\item For each tripotent $e\in E$ with $T(e) = 0$, the mapping $T$ is zero on $E_2(e)\oplus E_1(e)$.
\end{enumerate}
\end{lemma}

\begin{proof} The proof of $(a)$ is standard (compare \cite[Lemma 2.1]{BurFerGarPe2014preprint}). For the statement $(b)$, we recall that $a\perp b$ if and only if $\{a,a,b\}=0$ \cite[Lemma 1.1]{BurFerGarMarPe}. Let us consider the triple homomorphism $\Phi_{a,b}: E\to F$. Then $$\{T(a),T(a), T(b)\} = \{\Phi_{a,b}(a),\Phi_{a,b}(a), \Phi_{a,b}(b)\} = \Phi_{a,b} \{a,a,b\} = 0,$$ which proves $T(a) \perp T(b)$.\smallskip

$(c)$ Considering the triple homomorphism $\Phi_{a,a^{[3]}}$, we have $$\{T(a),T(a),T(a)\} = \{\Phi_{a,a^{[3]}}(a),\Phi_{a,a^{[3]}}(a),\Phi_{a,a^{[3]}}(a)\} $$ $$ = \Phi_{a,a^{[3]}} (a^{[3]}) = T(\{a,a,a\}).$$ The second statement follows from the polarization formula $$ 8\{x,y,z\}=\sum_{k=0}^3 \sum_{j=1}^2 i^k(-1)^j \left(x+i^ky+(-1)^jz\right)^{[3]}.$$

$(d)$ is clear from $(c)$, and $(e)$ follows from the fact that every triple homomorphism between JB$^*$-triples is contractive (cf. \cite[Lemma 1]{BarDan} and the proof of \cite[Lemma 2.1]{BurFerGarPe2014preprint}).\smallskip

$(f)$ Let us take a tripotent $e\in E$ with $T(e)$ a non-zero tripotent in $F$. For each $a\in E_{j} (e)$ we have $L(e,e) (a) = \frac{j}{2} a$. Therefore, $$L(T(e),T(e)) T(a) = \{T(e),T(e),T(a)\} = \{\Phi_{e,a} (e), \Phi_{e,a} (e),\Phi_{e,a} (a)\} $$ $$= \Phi_{e,a} (\{e,e,a\}) = \frac{j}{2} \Phi_{e,a} (a) = \frac{j}{2} T(a),$$ witnessing that $T(E_j (e)) \subseteq F_{j} (T(e))$, for every $j=0,1,2$. We can similarly\hyphenation{simi-larly} show that $T(a)\in \ker (Q(T(e))) =  F_0(T(e))+F_1(T(e))$ for every $a\in \ker (Q(e)) = E_0(e)+E_1(e)$ which shows that $T(E_0(e)+E_1(e))\subseteq F_0(T(e))+F_1(T(e)).$\smallskip

Since $F_2(T(e))+F_1(T(e))= \ker (P_0 (T(e)))$ and $$P_0 (T(e)) = (Id_{F}-L(T(e),T(e))) (Id_{F}-2 L(T(e), T(e))),$$ we can show, applying the triple homomorphism $\Phi_{a,e},$ that, for each element $a\in \ker (P_0 (e)) = E_2(e)+E_1(e),$ we have  $T(a)\in \ker (P_0 (T(e)))$, which gives the other inclusion.\smallskip

Suppose $a\in E_2 (e)_{sa}= \{x\in E_2 (e) : x= x^{\sharp_{e}}= \{e,x,e\}\}$. Since $$\{T(e), T(a), T(e)\} = \{\Phi_{e,a} (e), \Phi_{e,a} (a) , \Phi_{e,a} (e) \} $$ $$= \Phi_{e,a} \left(\{e,a,e\}\right) = \Phi_{e,a} (a) = T(a),$$ we deduce that $T(a) \in F_2 (T(e))_{sa}$.\smallskip

$(g)$ Suppose $T(e)=0$ and $a= a_1+a_2$, where $a_j\in E_{j} (e)$ for $j=1,2$. In such a case $$ T(a) = \Phi_{a,e} (a) = \Phi_{a,e} (\{e,e,a_2\}) + 2 \Phi_{a,e} (\{e,e,a_1\}) $$ $$=  \{\Phi_{a,e} (e),\Phi_{a,e}(e),\Phi_{a,e}(a_2) \} + 2  \{\Phi_{a,e} (e),\Phi_{a,e}(e),\Phi_{a,e}(a_1) \} $$ $$=\{T(e),T(e),\Phi_{a,e}(a_2) \} +  2  \{T(e),T(e),\Phi_{a,e}(a_1) \} = 0.$$
\end{proof}

We shall establish next a triple version of \cite[Lemma 3.1]{BurFerGarPe2014preprint}. 

\begin{lemma}\label{l linearity on single generated subalgebras}
Let $T: E\to F$ be a (not necessarily linear) 2-local triple homomorphism between JB$^*$-triples. Then, for each $a\in E$, $T|_{E_{a}} : E_{a} \to F$ is a linear mapping.
\end{lemma}

\begin{proof} 
Let us consider an element $b\in E_{a}$ of the form $\displaystyle b= \sum_{k=1}^{m} \alpha_k a^{[2k-1]}$ and the triple homomorphism $\Phi_{a,b}$. The identity $$T(b) = \Phi_{a,b} \left(\sum_{k=1}^{m} \alpha_k a^{[2k-1]}\right) = \sum_{k=1}^{m} \alpha_k \Phi_{a,b} \left( a\right)^{[2k-1]} = \sum_{k=1}^{m} \alpha_k T \left( a\right)^{[2k-1]}, $$ proves that $T$ is linear on the linear span of the set $\{ a^{[2k-1]} : k\in \mathbb{N}\}$. The continuity of $T$ shows that $T|_{E_{a}}$ is linear.
\end{proof}

Our next technical result establishes that every  (not necessarily linear) 2-local triple
homomorphism between JB$^*$-triples is additive on every couple of
orthogonal tripotents. The result is a generalization of \cite[Lemma 2.2]{BurFerGarPe2014preprint}
to the setting of JB$^*$-triples;
it should be noted that, in this more general setting, we need new and independent geometric arguments.

\begin{lemma}\label{l additivity on orthogonal tripotents}
Let $T: E\to F$ be a (not necessarily linear) 2-local triple homomorphism between JB$^*$-triples.
Let $e$ and $f$ be two orthogonal tripotents in $E$. Then $T(e + f) =
T(e) + T(f)$.
\end{lemma}

\begin{proof} Take a real number $\lambda\in (0,1]$. In this case we have
$$T(e+\lambda f)=\Phi_{e+\lambda f, e}(e+\lambda f)=T(e)+\lambda \Phi_{e+\lambda
f,e}(f)$$ with $\Phi_{e+\lambda f,e}(f)\perp T(e)$. We similarly have $$T(e+\lambda
f)=\Phi_{e+\lambda f, f}(e+\lambda f)=\Phi_{e+\lambda f, f}(e)+\lambda
T(f).$$

Combining the above identities we have that $$\Phi_{e+\lambda f,
f}(e)=T(e)+\lambda (\Phi_{e+\lambda
f,e}(f)-T(f))$$ $$=T(e)+P_0(T(e))\left(T(e+\lambda f)\right)-\lambda T(f).$$ Since $\Phi_{e+\lambda f, f}(e)$ and $T(e)$ are tripotents and $$T(e)\perp P_0(T(e))\left(T(e+\lambda f)\right)-\lambda T(f),$$
it follows that $P_0(T(e))\left(T(e+\lambda f)\right)-\lambda T(f)$ also is a tripotent for
every $\lambda \in (0,1]$.\smallskip

Clearly, the function $f:[0,1]\to \{0,1\}$
defined by $$f(\lambda):=\|P_0(T(e))T(e+\lambda f)-\lambda T(f)\|,$$
is continuous with $f(0)=0$, thus $f(\lambda)=0 \;\; \forall \lambda
\in [0,1]$. This implies, in particular, that $f(1)=0$, or equivalently,
$P_0(T(e))T(e+f)=T(f),$ which finishes the proof.
\end{proof}

The linearity of every (not necessarily linear) 2-local triple homomorphism on finite linear combinations
of mutually orthogonal tripotents follows next.

\begin{lemma}\label{l linearity on orthogonal tripotents}
Let $T: E\to F$ be a (not necessarily linear) 2-local triple homomorphism between
JB$^*$-triples. Let $e_1,\ldots, e_n$ be mutually orthogonal tripotents in $E$. Then
\begin{enumerate}[$(a)$]
\item $\displaystyle T\left(\sum_{i=1}^{n} e_i \right) =
    \sum_{i=1}^{n} T(e_i) $;
\item $\displaystyle  T\left(\sum_{i=1}^{n} \lambda_i e_i \right) = \sum_{i=1}^{n} \lambda_i T(e_i) $,
for every $\lambda_1, \ldots, \lambda_n\in \mathbb{C}$.
\end{enumerate}
\end{lemma}

\begin{proof} $(a)$ We shall argue by induction on $n$. The case $n=1$ is clear, while the case $n=2$ is established in Lemma \ref{l additivity on orthogonal tripotents}. Let us suppose that $e_1,\ldots, e_n, e_{n+1}$ are mutually orthogonal tripotents in $E$. Since $e = e_1 +\ldots +e_n$ and $e_{n+1}$ are orthogonal tripotents in $E$, Lemma \ref{l additivity on orthogonal tripotents} and the induction hypothesis prove that $$T\left(\sum_{i=1}^{n+1} e_i \right) = T (e + e_{n+1} ) = T(e) + T(e_{n+1}) = \sum_{i=1}^{n} T(e_i)  + T(e_{n+1}).$$

$(b)$  Fix $j\in \{ 1,\ldots, n\}$ and set $z= \displaystyle \sum_{i=1}^{n} \lambda_i e_i$ and  $e=\displaystyle \sum_{i=1}^{n} e_i$. The identity $$ \left\{ T(e_j), T(e_j) ,T\left(\sum_{i=1}^{n} \lambda_i e_i \right) \right\} = \left\{ \Phi_{z,e_j}(e_j), \Phi_{z,e_j}(e_j) ,\Phi_{z,e_j}\left(\sum_{i=1}^{n} \lambda_i e_i \right) \right\}$$ $$= \Phi_{z,e_j} \left( \left\{e_j,e_j, \sum_{i=1}^{n} \lambda_i e_i \right\} \right)  = \Phi_{z,e_j} (\lambda_j e_j ) = \lambda_j T(e_j),$$ implies that $$T\left(\sum_{i=1}^{n} \lambda_i e_i \right) = \Phi_{z,e} \left( z \right)    = \Phi_{z,e} \left(\left\{e,e,z\right\} \right) = \left\{\Phi_{z,e} (e),\Phi_{z,e} (e),\Phi_{z,e} \left(z \right)\right\} $$ $$= \left\{T (e),T (e),T \left(z \right)\right\}  = \left\{T \left(\sum_{j=1}^{n} e_j\right),T \left(\sum_{j=1}^{n} e_j\right),T \left(\sum_{i=1}^{n}  \lambda_i e_i \right)\right\}$$ $$=\hbox{(by $(a)$)}= \left\{\sum_{j=1}^{n} T \left(e_j\right),\sum_{j=1}^{n} T \left(e_j\right),T \left(\sum_{i=1}^{n}  \lambda_i e_i \right)\right\}=\hbox{(by orthogonality)} $$ $$= \sum_{j=1}^{n} \left\{ T \left(e_j\right),T \left(e_j\right),T \left(\sum_{i=1}^{n}  \lambda_i e_i \right)\right\} = \sum_{j=1}^{n}  \lambda_j T(e_j).$$
\end{proof}

Let $E$ and $F$ be JB$^*$-triples. We recall that a (not necessarily linear) mapping $f: E\to F$ is called \emph{orthogonally additive} if $f(a+b) = f(a)+f(b)$ for every $a\perp b$ in $E$.

\begin{proposition}\label{p 2-local triple homomorphisms are OA}
Let $E$ be a JBW$^*$-triple, $F$ a JB$^*$-triple,
and suppose that $T : E \to F$ is a (not necessarily linear)
2-local triple homomorphism. Then $T$ is orthogonally additive.
\end{proposition}

\begin{proof}
Let $a$ and $b$ be two orthogonal elements in $E$. The range tripotents $r(a)$ and $r(b)$ are orthogonal, and the JBW$^*$-subtriples $E_2 (r(a))$ and $E_2 (r(b))$ are also orthogonal (cf. \cite[Lemma 1.1]{BurFerGarMarPe}).\smallskip

For each $\varepsilon >0$, there exists two algebraic elements $\displaystyle a_\varepsilon = \sum_{k= 1}^{m_1} \lambda_k e_k $ and $b_{\varepsilon}= \sum_{j= 1}^{m_2} \mu_j v_j$, where $\lambda_k,\mu_j\in \mathbb{R}$, $e_1,\ldots, e_{m_1}$ and $v_1,\ldots, v_{m_2}$ are tripotents in $E_2 (r(a))$ and $E_2 (r(b))$, respectively, and  $e_j\perp e_k,$ $v_j\perp v_k$ for every $j\neq k$, such that
$\|a-a_{\varepsilon}\| < \frac{\varepsilon}{4}$, and $\|b-b_{\varepsilon}\| < \frac{\varepsilon}{4}$ (cf. \cite[lemma 3.11]{Horn}). It is clear that
$a_{\varepsilon}+b_{\varepsilon}$ is a linear combination of mutually orthogonal tripotents. Then, by Lemma \ref{l basic
properties}$(e)$ and Lemma \ref{l linearity on orthogonal tripotents}$(b)$,
$$\|T(a+b)-T(a)-T(b)\|=\|T(a+b)-T(a_{\varepsilon}+b_{\varepsilon})+T(a_{\varepsilon})+T(b_{\varepsilon})-T(a)-T(b)\|$$ $$\leq \|T(a+b)
-T(a_{\varepsilon}+b_{\varepsilon})\| + \|T(a_{\varepsilon})- T(a)
\|$$ $$ +\|T(b_{\varepsilon})-T(b) \| < \| (a+b)
-(a_{\varepsilon}+b_{\varepsilon})\| + \|a_{\varepsilon}- a \|
+\|b_{\varepsilon}-b \|< \varepsilon. $$ Since $\varepsilon$ was
arbitrarily chosen, we get $T(a+b) = T(a) + T(b).$
\end{proof}

A simple induction argument, combined with Proposition \ref{p 2-local triple homomorphisms are OA}, shows:

\begin{corollary}\label{c stability for ell infty sums}
Let $\left(E_i\right)_{i=1}^{n}$ be a finite family of
JBW$^*$-triples and let $F$ be a JB$^*$-triple. Suppose that, for
every $i$, every (not necesarily linear) 2-local triple homomorphism $T: E_i \to
F$ is linear. Then every (not necesarily linear) 2-local triple homomorphism $\displaystyle
T: \bigoplus_{i=1,\ldots,n}^{\ell_{\infty}} E_i \to F$
is linear. $\hfill \Box$
\end{corollary}

We recall now a result, due to Friedmann and Russo, which has been borrowed from \cite[Lemma 1.6]{FriRu85}.

\begin{lemma}\label{l FriRu 1.6} \cite[Lemma 1.6]{FriRu85} Let $e$ be a tripotent in a JB$^*$-triple $E$. Then, for each norm-one element $x\in E$ satisfying $P_2(e)x=e$, we have $P_1(e) x=0$. $\hfill\Box$
\end{lemma}

In order to make the results more accessible, we have splitted the technical arguments needed in the proofs of our main results into a series of lemmas and propositions, which assure certain almost-linearity properties of 2-local triple homomorphisms.

\begin{lemma}\label{l triple hom with z open unit ball} Let $T : E \to F$ be a (not necessarily linear) 2-local triple homomorphism between JB$^*$-triples. Suppose $e$ is a tripotent in $E$ and $z\in E_0 (e)$ with $\|z\|<1$. Then, for each $w\in E$ and each triple homomorphism $\Phi_{w,e+z} : E\to F$ satisfying $\Phi_{w,e+z} (w) = T(w)$ and $\Phi_{w,e+z} (e+z) = T(e+z)$, we have $\Phi_{w,e+z} (e) = T (e)$, and consequently, $\Phi_{w,e+z} (E_j (e))\subseteq F_j (T(e))$, for every $j=0,1,2$.
\end{lemma}

\begin{proof} By considering the triple homomorphism $\Phi_{e, e+z}$, we obtain that $$T(e+z) = \Phi_{e, e+z} (e+z) = T(e) + \Phi_{e, e+z} (z),$$ where $\Phi_{e, e+z} (z)\in F_{0} (\Phi_{e, e+z} (e)) = F_{0} (T(e))$ and $\|\Phi_{e, e+z} (z) \|\leq \|z\|<1$. Since $\|z\|<1$, $\|\Phi_{e, e+z} (z)\|<1$ and $T(z)\in F_0 (T(e))$ (cf. Lemma \ref{l basic properties}), we have,
$$ \Phi_{w,e+z}(e)=\Phi_{w,e+z}\left(\lim_{n\to \infty} (e+z)^{[3^n]}\right)=\lim_{n\to \infty} \left(\Phi_{w,e+z}(e+z)\right)^{[3^n]}$$
$$= \lim_{n\to \infty} \left(T(e+z)\right)^{[3^n]}=\lim_{n\to \infty} \left(T(e)+\Phi_{e, e+z} (z)\right)^{[3^n]}=T(e),$$ where all
the above limits are in the norm-topology.
\end{proof}

\begin{lemma}\label{l technical on tripotents}
Let $T: E\to F$ be a (not necessarily linear) 2-local triple homomorphism between JB$^*$-triples.
Then the following statements hold:
\begin{enumerate}[$(a)$]
\item  For each tripotent $e$ in $E$, and each $y\in E_1(e)$, we have $$T(e+y)=T(e)+T(y);$$
\item  Suppose $e_1,e_2,$ and $g$ are tripotents in $E$ satisfying $e_1\perp e_2$, $e_1,e_2
    \in E_2(g)$, $g \in E_1(e_1)\cap E_1(e_2)$. Then, the identity
    $$T(\lambda_1 e_1+\mu g+\lambda_2 e_2)=\lambda_1 T(e_1)+\mu
    T(g)+\lambda_2 T(e_2),$$ holds for every $\lambda_1, \lambda_2, \mu \in \mathbb{C}$.
\end{enumerate}
\end{lemma}

\begin{proof} $(a)$ Let $e$ be a tripotent in $E$, and let $y\in E_1(e)$.
By Lemma \ref{l basic properties}$(g)$, the desired statement is clear when $T(e) = 0$,
so we assume that $T(e) \neq 0$. In this case,
$$T(e+y)=\Phi_{e+y,e}(e+y)=T(e)+\Phi_{e+y,e}(y),$$ where
$\Phi_{e+y,e}(y)\in F_1(\Phi_{e+y,e}(e))=F_1(T(e))$, and we also have
$$T(e+y)=\Phi_{e+y,y}(e+y)=\Phi_{e+y,y}(e)+T(y),$$ with $\Phi_{e+y,y}(e)$
being a tripotent. Therefore, $$\Phi_{e+y,y}(e)=T(e)+ \Phi_{e+y,e}(y)-T(y),$$  with $\Phi_{e+y,e}(y)-T(y)\in F_1(T(e))$.
It follows from Lemma \ref{l FriRu 1.6} that $$0=P_1(T(e))(\Phi_{e+y,y}(e))=\Phi_{e+y,e}(y)-T(y),$$ witnessing the desired statement.\smallskip

$(b)$ We can assume that $\lambda_1, \lambda_2, \mu \neq 0$, otherwise the statement is clear from $(a)$ or from Lemma \ref{l linearity on orthogonal tripotents}. To simplify notation, we set $z=\lambda_1 e_1+\mu g+\lambda_2 e_2$.
Applying $(a)$ we get $$T(z)=\Phi_{z,\lambda_1 e_1+\mu g}(z)=\lambda_1 T(e_1)+\mu T(g)+\lambda_2 \Phi_{z,\lambda_1 e_1+\mu g}(e_2).$$ We also have
\begin{equation}\label{eq 1 lem tech 1 new} T(z)=\Phi_{z,e_2}(z)=\lambda_1 \Phi_{z,e_2}(e_1)+\mu \Phi_{z,e_2}(g)+\lambda_2 T(e_2).
\end{equation} Combining these two equalities we have
$$\Phi_{z,\lambda_1 e_1+\mu g}(e_2)=T(e_2)+\frac{\mu}{\lambda_2}(\Phi_{z,e_2}(g)-T(g))+\frac{\lambda_1}{\lambda_2}(\Phi_{z,e_2}(e_1)-T(e_1)),$$
where $\Phi_{z,e_2}(e_1)\in F_0 (\Phi_{z,e_2} (e_2)) = F_0 (T(e_2))$, $T(e_1)\in F_0 (T(e_2))$, $\Phi_{z,e_2}(g)\in F_1(\Phi_{z,e_2} (e_2))=F_1(T(e_2))$, and  $T(g) \in F_1(T(e_2))$ (cf. Lemma \ref{l basic properties}$(g)$). Lemma \ref{l FriRu 1.6} implies that $T(g)=\Phi_{z,e_2}(g),$ and hence \eqref{eq 1 lem tech 1 new} writes in the form $$T(z)=\Phi_{z,e_2}(z)=\lambda_1 \Phi_{z,e_2}(e_1)+\mu T(g)+\lambda_2 T(e_2).$$ The last identity implies that $P_2 (T(e_2)) T(z) = \lambda_2 T(e_2)$, $P_1 (T(e_2)) T(z) = \mu T(g)$, and $P_0 (T(e_2)) T(z) = \lambda_1 \Phi_{z,e_2}(e_1)$.\smallskip

The identity $$T(z) = \Phi_{z,e_1} (z) = \lambda_1 T(e_1) + \mu \Phi_{z,e_1} (g)+ \lambda_2 \Phi_{z,e_1} (e_2),$$ shows that
$P_2 (T(e_1)) T(z) = \lambda_1 T(e_1).$\smallskip

Finally, having in mind that $T(z) \in F_2 (T(e_1+e_2))=F_2 (T(e_1)+T(e_2))$, and $F_0(T(e_2))\cap F_2(T(e_1+e_2))=F_2(T(e_1))$ (cf. \cite[1.12]{Horn}), we have $\Phi_{z,e_2}(e_1)=T(e_1)$.
\end{proof}

\begin{lemma}\label{l 2-local and Peirce arithmetics}
Let $T: E\to F$ be a (not necessarily linear) 2-local triple homomorphism from a JBW$^*$-triple into a JB$^*$-triple, and let $e$ be a tripotent in $E$.
Then the following statements hold:
\begin{enumerate}[$(a)$]
\item $T(e+y+z)=T(e)+T(y)+T(z)$, for every $y\in E_1(e)$, and every $z\in E_0(e)$ with $\|z\|<1;$
\item $T(y+z)=T(y)+T(z)$, for every $y\in E_1(e)$, and every $z\in E_0(e)$;
\item $T(e+y+z)=T(e)+T(y)+T(z)$, for every $y\in E_1(e)$, and every $z\in E_0(e)$. Consequently, $T(\lambda e+y+z)=\lambda T(e)+T(y)+T(z)$, for every $y\in E_1(e)$, every $z\in E_0(e)$, and every $\lambda\in \mathbb{C}$.
\end{enumerate}
\end{lemma}

\begin{proof} Throughout the proof we set $w=e+y+z$.\smallskip

$(a)$ We assume first that $T(e)=0$. By Lemma \ref{l basic properties}$(g)$, $T(y)=0$. By Lemma \ref{l triple hom with z open unit ball}, $\Phi_{w,e+z} (e) = T(e)=0$ and hence $\Phi_{w,e+z}(y) \in F_1 (\Phi_{w,e+z}(e)) =\{0\}$. If we write $$T(e+y+z) = T(w)=\Phi_{w,e+z}(w)=T(e+z)+\Phi_{w,e+z}(y)$$ $$=T(e+z) =\hbox{(by Proposition \ref{p 2-local triple homomorphisms are OA})} = T(e) + T(z) =0.$$

\noindent Suppose now that $T(e)\neq 0$. Proposition \ref{p 2-local triple homomorphisms are OA} implies that
$$T(w)=\Phi_{w,e+z}(w)=T(e+z)+\Phi_{w,e+z}(y)=T(e)+T(z)+\Phi_{w,e+z}(y).$$ By Lemma \ref{l triple hom with z open unit ball}, $\Phi_{w,e+z} (e) = T(e)$, and
in particular $\Phi_{w,e+z}(y)\in F_1(T(e))$. We also have
$$T(w)=\Phi_{w,y}(w)=T(y)+\Phi_{w,y}(e+z),$$ and hence $$\Phi_{w,y}(e+z)=T(e)+ \Phi_{w,e+z}(y)-T(y)+
T(z).$$ Having in mind that $\|\Phi_{w,y}(e+z)\|\leq 1$, Lemma \ref{l FriRu 1.6} implies that
$$0=P_1(T(e))\Phi_{w,y}(e+z)=\Phi_{w,e+z}(y)-T(y).$$

$(b)$ Since $T$ is 1-homogeneous, we may assume without loss of generality that $\|z\|<1$. As in the previous case, let us assume that $T(e)=0$. Under these assumptions, Lemma \ref{l triple hom with z open unit ball} implies that $\Phi_{y+z,e+z} (e+y+z-e) (e) = T(e) = 0$ and hence $\Phi_{y+z,e+z} (E_2 (e)\oplus E_1 (e))= \{0\}.$ Then $$T(y+z) = \Phi_{y+z,e+z} (e+y+z-e) = T(e+z) + \Phi_{y+z,e+z} (y-e)$$
$$= T(e+z) = \hbox{(by Proposition \ref{p 2-local triple homomorphisms are OA})} = T(e) +T(z) = T(z).$$

\noindent We consider now the case $T(e)\neq 0$. Since we are assuming $\|z\|<1$, it follows from $(a)$ that
$$T(y+z)=\Phi_{y+z,w}(e+y+z-e)=\Phi_{y+z,w}(e+y+z)-\Phi_{y+z,w}(e)$$ $$=T(e)+T(y)+T(z)-\Phi_{y+z,w}(e).$$
Considering that $T(y+z)\in F_1(T(e))+F_0(T(e))$ (see Lemma \ref{l basic properties}$(f)$) we have $P_2(T(e))\Phi_{y+z,w}(e)=T(e)$.
Lemma \ref{l FriRu 1.6}, applied in the identity $\Phi_{y+z,w}(e) = T(e)+T(y)- T(y+z) +T(z)$, shows that $P_1(T(e))T(y+z)=T(y)$. \smallskip

By Corollary \ref{p 2-local triple homomorphisms are OA}, we get $$T(y+z)=\Phi_{y+z,e+z}(e+z)+\Phi_{y+z,e+z}(y-e)=T(e+z)+\Phi_{y+z,e+z}(y-e)$$
    $$=T(e)+T(z)+\Phi_{y+z,e+z}(y)-\Phi_{y+z,e+z}(e).$$

By assumptions $\|z\|<1$. Thus, applying Lemma \ref{l triple hom with z open unit ball} we show $\Phi_{y+z,e+z}(e)=T(e).$ Therefore, $\Phi_{y+z,e+z}(y)\in F_1(T(e))$ and $P_0(T(e))T(y+z)=T(z)$, which proves that $$T(y+z) = P_1(T(e))T(y+z)+ P_0(T(e))T(y+z)=T(y)+T(z).$$

$(c)$ We begin with the case $T(e) \neq 0$. For each real number $\lambda\in [0,1]$, we denote $w_{\lambda}:=e+y+\lambda z$.
By the assumptions on $T$
$$T(w_{\lambda})=\Phi_{w_{\lambda},e}(w_{\lambda})=T(e)+\Phi_{w_{\lambda},e}(y)+\lambda \Phi_{w_{\lambda},e}(z),$$
where $\Phi_{w_{\lambda},e}(y)\in F_1(T(e)),$ and $\Phi_{w_{\lambda},e}(z)\in F_0(T(e))$. Applying $(b)$ we deduce that
$$T(w_{\lambda})=\Phi_{w_{\lambda},y+\lambda z}(w_{\lambda})=T(y)+\lambda
    T(z)+\Phi_{w_{\lambda},y+\lambda z}(e).$$
The above identities show that $$\Phi_{w_{\lambda},y+\lambda z}(e)=T(e)+(\Phi_{w_{\lambda},e}(y)-T(y))+\lambda (\Phi_{w_{\lambda},e}(z)-T(z)),$$ and Lemma \ref{l FriRu 1.6} applies to assure that $\Phi_{w_{\lambda},e}(y)=T(y)$. Therefore,
$$\Phi_{w_{\lambda},y+\lambda z}(e)=T(e)+\lambda (\Phi_{w_{\lambda},e}(z)-T(z)).$$
Since $\Phi_{w_{\lambda},y+\lambda z}(e)$ is a tripotent, we deduce that $$P_0(T(e))\Phi_{w_{\lambda},y+\lambda z}(e)=\lambda
    (\Phi_{w_{\lambda},e}(z)-T(z))=P_0(T(e))(T(w_{\lambda}))-\lambda T(z)$$ is a tripotent.\smallskip

Finally the mapping $f:[0,1]\to \{0,1\}$ given by
$$f(\lambda)=\| P_0(T(e))(T(w_{\lambda}))-\lambda T(z)\|$$ is (norm) continuous and $f(0)=0$,
then $f(\lambda)=0$ for every $\lambda \in [0,1]$, and hence $\Phi_{w_{1},e}(z)=T(z),$ which gives the desired statement.\smallskip

Suppose, finally, that $T(e)=0$. Lemma \ref{l basic properties}$(g)$ implies that $T(y) =0$. Let us observe that $\Phi_{w_{\lambda},e}(e)=T(e) =0$.
The identities
$$T(w_{\lambda})=\Phi_{w_{\lambda},e}(w_{\lambda})=T(e)+\Phi_{w_{\lambda},e}(y)+\lambda \Phi_{w_{\lambda},e}(z) = \lambda \Phi_{w_{\lambda},e}(z),$$
$$T(w_{\lambda})=\Phi_{w_{\lambda},y+\lambda z}(w_{\lambda})=\Phi_{w_{\lambda},y+\lambda z} (e)+ T(y)+\lambda T(z) = \Phi_{w_{\lambda},y+\lambda z}(e)+ \lambda T(z),$$ show that $$\Phi_{w_{\lambda},y+\lambda z}(e)=  \lambda \Phi_{w_{\lambda},e}(z) - \lambda T(z)  = T(w_{\lambda}) -\lambda T(z), $$ for every $\lambda\in [0,1]$. Since, for every $0\leq \lambda \leq 1$, $\Phi_{w_{\lambda},y+\lambda z}(e)$ is a tripotent, the function $f: [0,1] \to \mathbb{R}$, $f(\lambda):= \left\| \Phi_{w_{\lambda},y+\lambda z}(e) \right\| = \left\| T(w_{\lambda}) -\lambda T(z) \right\|$ is continuous and takes only the values $0$ and $1$. Since $f(0)=0$, we conclude that $f(\lambda) = 0$, for every $\lambda\in [0,1]$, which proves $T(e+y+z) - T(z) = 0$.
\end{proof}

We recall that, given a conjugation (conjugate linear isometry of period 2), $\sigma$, on a
complex Hilbert space $H$ with dim$(H)=n\in \mathbb{N}\cup \{\infty\}$, the mapping $x \mapsto x^{t}:=\sigma x^* \sigma$
defines a linear involution on $L(H)$. The type-3 Cartan
factor, denoted by $III_{n}$, is the subtriple of $L(H)$ formed by
the $t$-symmetric operators. Following standard notation, $S_2 (\mathbb{C})$ will denote $III_{2}$.

\begin{corollary}\label{c 2-local triple hom on M2 and S2}
Let $F$ be a JB$^*$-triple and let $T: C \to F$ be a (not
necessarily linear) 2-local triple-homomorphism, where $C$ is $M_2
(\mathbb{C})$ or $S_2 (\mathbb{C})$. Then $T$ is linear and a
triple homomorphism.
\end{corollary}

\begin{proof} Suppose first that $C= M_2(\mathbb{C})$. We set $e_1=\left(
                                                                   \begin{array}{cc}
                                                                     1 & 0 \\
                                                                     0 & 0 \\
                                                                   \end{array}
                                                                 \right)$,  $e_2=\left(
                                                                   \begin{array}{cc}
                                                                     0 & 0 \\
                                                                     0 & 1 \\
                                                                   \end{array}
                                                                 \right)$, $y_1=\left(
                                                                   \begin{array}{cc}
                                                                     0 & 1 \\
                                                                     0 & 0 \\
                                                                   \end{array}
                                                                 \right)$, and $y_2=\left(
                                                                   \begin{array}{cc}
                                                                     0 & 0 \\
                                                                     1 & 0 \\
                                                                   \end{array}
                                                                 \right)$. Lemma \ref{l 2-local and Peirce arithmetics}$(c)$ implies that
$$T(\lambda_1 e_1 + \mu_1 y_1 + \mu_2 y_2 + \lambda_2 e_2) = \lambda_1 T(e_1) + T(\mu_1 y_1 + \mu_2 y_2) + \lambda_2 T(e_2)$$
$$=\hbox{(Proposition \ref{p 2-local triple homomorphisms are OA} applied to $y_1\perp y_2$)} =$$ $$= \lambda_1 T(e_1) + \mu_1 T(y_1) + \mu_2 T(y_2) + \lambda_2 T(e_2).$$ The linearity follows from the fact that $\{e_1,y_1,y_2,e_2\}$ is a basis of $M_2 (\mathbb{C})$.\smallskip

For the statement concerning $S_2 (\mathbb{C})$, we observe that we can assume that $\sigma: H =\ell_2^2 \to H=\ell_2^2$ is the mapping given by $\sigma (t_1,t_2) =  (\overline{t}_1,\overline{t}_2).$ Considering $e_1=\left(\begin{array}{cc}
                                                                     1 & 0 \\
                                                                     0 & 0 \\
                                                                   \end{array}
                                                                 \right)$,  $e_2=\left(
                                                                   \begin{array}{cc}
                                                                     0 & 0 \\
                                                                     0 & 1 \\
                                                                   \end{array}
                                                                 \right)$, and $y=\left(
                                                                   \begin{array}{cc}
                                                                     0 & 1 \\
                                                                     1 & 0 \\
                                                                   \end{array}
                                                                 \right)$, Lemma \ref{l 2-local and Peirce arithmetics}$(c)$ implies that
$$T(\lambda_1 e_1 + \mu y + \lambda_2 e_1) =  \lambda_1 T(e_1) + \mu T(y) + \lambda_2 T(e_2),$$ which proves that $T$ is linear.
\end{proof}

\begin{proposition}\label{p 2-local in Peirce-2 Peirce-1 arithmetics}
Let $T: E\to F$ be a (not necessarily linear) 2-local triple homomorphism from a JBW$^*$-triple into a JB$^*$-triple, and let $e$ be a tripotent in $E$.
Then $T(x+y)=T(x)+T(y),$ for all $x\in E_2(e)$, $y\in E_1(e)$.
\end{proposition}

\begin{proof} Let us observe that by Lemma \ref{l basic properties}$(g)$, we may assume that $T(e)\neq 0$. By the norm density of algebraic elements in $E_2(e)$ (cf. \cite[lemma 3.11]{Horn}), together with the continuity of $T$, it is enough to prove that, for every algebraic element $a$ in $E_2(e)$ (i.e. $\displaystyle a = \sum_{k= 1}^{m} \lambda_k e_k $, where $\lambda_k\in \mathbb{R}$, and  $e_1,\ldots, e_{m_1}$ are mutually orthogonal tripotents in $E_2 (e)$), we have $T(a+y)=T(a)+T(y)$. We shall prove this statement by induction on the
number $m$ of mutually orthogonal tripotents whose linear combination coincides with $a$.\smallskip

For the case $m=1$, we may assume that $a=\lambda_1 e_1 \in E_2(e)$ with $\lambda_1\neq 0$.
Since $y\in E_1(e)$, it follows from Peirce rules that $y$ writes in the form $y=y_1+y_0,$ where $y_k=P_k(e_1)y$,
$k=1,0$. By Lemma \ref{l 2-local and Peirce arithmetics}$(c)$, $$T(a+y)=T(\lambda_1 e_1+y_1+y_0)=T(\lambda_1
e_1)+T(y_1)+T(y_0),$$ and by Lemma \ref{l 2-local and Peirce
arithmetics}$(b)$, $T(y_1)+T(y_0)=T(y_1+y_0)=T(y)$, then $T(a+y)=T(a)+T(y)$.\smallskip

Suppose, by the induction hypothesis, that for every algebraic element $b$ in $E_2(e)$
which is a linear combination of $m$ mutually orthogonal tripotents in $E_2(e)$, we have $$T(b+y)=T(b)+T(y),$$ for
every $y\in E_1(e)$. Let $\displaystyle a=\sum_{i=1}^{m+1} \lambda_i e_i$ be an algebraic element in
$E_2(e)$, and denote by $f$ the tripotent $\displaystyle
\sum_{i=1}^{m+1} e_i$. Applying the Peirce decompositions of $a+y$ associated with $f$ and $e_1$, we have
$a+y=a+y_1+y_0,$ where $y_1=P_1(f)y$ and $y_0=P_0(f)y$, and
$$a+y=\lambda_1 e_1 + P_1(e_1)y_1+\left(\sum_{i=2}^{m+1} \lambda_i
e_i+P_0(e_1)y_1+y_0\right),$$ where $\displaystyle \left(\sum_{i=2}^{m+1}
\lambda_i e_i+P_0(e_1)y_1+y_0\right)\in E_0(e_1)$.\smallskip

Lemma \ref{l 2-local and Peirce arithmetics}$(c)$ implies that
\begin{equation}\label{eq 1 in prop Peirce 2 and 1} T(a+y)=T(\lambda_1 e_1)+T(P_1(e_1)y_1)+T\left(\sum_{i=2}^{m+1} \lambda_i
e_i+P_0(e_1)y_1+y_0\right).
\end{equation} We observe that $e_1, f\in E_2 (e)$, therefore $P_j (e_1) P_k (e) = P_k(e) P_j (e_1)$ and $P_j (f) P_k (e) = P_k(e) P_j (f)$, for every $j,k\in \{0,1,2\}$ (cf. \cite[Lemma 1.10]{FriRu85}). The induction hypothesis, applied to
$\displaystyle \sum_{i=2}^{m+1} \lambda_i e_i\in E_2 (e)$ and $P_0(e_1)y_1 +y_0 = P_0(e_1) P_1 (f) (y) + P_0(f) (y) = P_0(e_1) P_1 (f) P_1 (e)(y) + P_0(f) P_1 (e)(y)= P_1(e) (P_0(e_1) P_1 (f) (y) + P_0(f) (y))\in E_1(e),$ assures that \begin{equation}\label{eq 2 in prop Peirce 2 and 1}T\left(\sum_{i=2}^{m+1} \lambda_i
e_i+P_0(e_1)y_1+y_0\right)=T\left(\sum_{i=2}^{m+1} \lambda_i
e_i\right)+T(P_0(e_1)y_1+y_0).\end{equation}

Finally, by Lemma \ref{l linearity on orthogonal tripotents} \begin{equation}\label{eq 3 in prop Peirce 2 and 1} \displaystyle T(\lambda_1 e_1)+T\left(\sum_{i=2}^{m+1} \lambda_i
e_i\right)=T(a).
 \end{equation} Since $P_1(e_1)y_1 \in E_1(e_1),$ $ P_0(e_1)y_1+y_0\in E_0(e_1),$ Lemma \ref{l 2-local and
Peirce arithmetics} $(b)$ implies that
$$T(P_1(e_1)y_1)+T(P_0(e_1)y_1+y_0)=T(P_1(e_1)y_1+P_0(e_1)y_1+y_0)$$ $$=T(y_1+y_0)=T(y),$$ which combined with \eqref{eq 1 in prop Peirce 2 and 1}, \eqref{eq 2 in prop Peirce 2 and 1}, and \eqref{eq 3 in prop Peirce 2 and 1} prove $T(a+y)=T(a)+T(y).$
\end{proof}

Our series of technical results on 2-local triple homomorphisms concludes with an strengthened version of Lemma \ref{l 2-local and Peirce arithmetics}.

\begin{lemma}\label{l 2-local and Peirce arithmetics II}
Let $T: E\to F$ be a (not necessarily linear) 2-local triple homomorphism from a JBW$^*$-triple into a JB$^*$-triple, and let $e$ be a tripotent in $E$.
Then the following statements hold:
\begin{enumerate}[$(a)$]
\item $T(e+i h+y+z)=T(e)+ i T(h)+T(y)+T(z)$, for every $h\in E_2 (e)_{sa}$, $y\in E_1(e)$, and every $z\in E_0(e)$ with $\|z\|<1;$
\item $T(i h +y+z)=i T(h)+T(y)+T(z)$, for every $h\in E_2 (e)_{sa}$, $y\in E_1(e)$, and every $z\in E_0(e)$;
\item $T(e+ i h+y+z)=T(e)+i T(h)+T(y)+T(z)$, for every $h\in E_2 (e)_{sa}$, $y\in E_1(e)$, and every $z\in E_0(e)$. Consequently, $$T(\lambda e+ i h+y+z)=\lambda T(e)+i T(h)+T(y)+T(z),$$ for every $\lambda\in \mathbb{C}$, $h\in E_2 (e)_{sa}$, $y\in E_1(e)$, and every $z\in E_0(e)$.
\end{enumerate}
\end{lemma}

\begin{proof} Along this proof we set $w=e+i h+y+z$.\smallskip

\noindent We shall assume first that $T(e)= 0$. By Lemma \ref{l basic properties}$(g)$, we have $T(ih) = T(y)=0$.\smallskip

\noindent $(a)$ It follows from Lemma \ref{l triple hom with z open unit ball} that $\Phi_{w,e+z} (e) = T(e)=0$, and hence $\Phi_{w, e+z} (E_{2} (e) \oplus E_1 (e))=\{0\}$. Therefore, $$T(e+i h+y+z)=\Phi_{w,e+z}(w)=T(e+z)+i \Phi_{w,e+z}(h)+\Phi_{w,e+z}(y)$$
$$= \hbox{(by Proposition \ref{p 2-local triple homomorphisms are OA})}=T(e)+T(z)+i \Phi_{w,e+z}(h)+\Phi_{w,e+z}(y) = T(z).$$

$(b)$ Since $T$ is $1$-homogeneous, we may assume, without losing generality, that $\|z\|<1$. Lemma \ref{l triple hom with z open unit ball} implies that $\Phi_{w-e,e+z} (e) = T(e)=0$. We write $$T(ih +y+z) = \Phi_{w-e,e+z}(w-e)=T(e+z)+i \Phi_{w-e,e+z}(h)+\Phi_{w-e,e+z}(y) $$
$$ = T(e+z)= \hbox{(by Proposition \ref{p 2-local triple homomorphisms are OA})}=T(e)+T(z)=T(z).$$\smallskip

$(c)$ Given $\lambda\in [0,1]$, we set $w_{\lambda}:=e+ i h+y+\lambda z$. The identities:
$$T(w_{\lambda})=\Phi_{w_{\lambda},e}(w_{\lambda})=T(e)+ i \Phi_{w_{\lambda},e}(h)+\Phi_{w_{\lambda},e}(y)+\lambda \Phi_{w_{\lambda},e}(z)= \lambda \Phi_{w_{\lambda},e}(z),$$ and
$$T(w_{\lambda})=\Phi_{w_{\lambda},ih+y+\lambda z}(w_{\lambda}) = \Phi_{w_{\lambda},ih+y+\lambda z} (e) + T(ih+y+\lambda z)$$
$$=\hbox{(by $(b)$)}= \Phi_{w_{\lambda},ih+y+\lambda z} (e) + iT(h) +T(y) +\lambda T(z)=  \Phi_{w_{\lambda},ih+y+\lambda z} (e) +\lambda T(z)$$ assure that
$$ \Phi_{w_{\lambda},ih+y+\lambda z} (e) = T(w_{\lambda}) - \lambda T(z).$$ Arguing as in the proof of Lemma \ref{l 2-local and Peirce arithmetics}$(c)$ (case $T(e) =0$), we deduce that $T(w_{\lambda}) = \lambda T(z)$, for every $0\leq \lambda\leq 1$.
\medskip

\noindent We suppose from this moment that $T(e)\neq 0$.\smallskip

\noindent $(a)$ We begin with the identity $$T(w)=\Phi_{w,e+z}(w)=T(e+z)+i \Phi_{w,e+z}(h)+\Phi_{w,e+z}(y)$$
$$= \hbox{(by Proposition \ref{p 2-local triple homomorphisms are OA})}=T(e)+T(z)+i \Phi_{w,e+z}(h)+\Phi_{w,e+z}(y).$$
We deduce, by Lemma \ref{l triple hom with z open unit ball}, that $\Phi_{w,e+z}(e)=T(e)$, which, in particular, gives $\Phi_{w,e+z}(y)\in
F_1(T(e))$ and $\Phi_{w,e+z}(h)\in F_2 (T(e))_{sa}$.\smallskip

On the other hand,
$$T(w)=\Phi_{w,ih +y}(w)=T(i h+ y)+\Phi_{w,ih +y}(e+z)$$ $$= \hbox{(by Proposition \ref{p 2-local in Peirce-2 Peirce-1 arithmetics})} = i T( h) + T( y)+\Phi_{w,ih +y}(e+z),$$ and hence $$\Phi_{w,ih +y}(e+z) = T(e)+ i \left(\Phi_{w,e+z}(h)-T(h)\right)+ \left(\Phi_{w,e+z}(y)-T(y)\right)+
T(z).$$ The element $T(e)+ i \left(\Phi_{w,e+z}(h)-T(h)\right)$ lies in the JB$^*$-algebra $F_2 (T(e))$ and $\left(\Phi_{w,e+z}(h)-T(h)\right)\in F_2 (T(e))_{sa}$ (cf. Lemma \ref{l basic properties}$(f)$) with $$\left\|T(e)+ i \left(\Phi_{w,e+z}(h)-T(h)\right)\right\| \leq \|\Phi_{w,ih +y}(e+z) \|\leq 1.$$ Therefore $$1\geq\left\|T(e)+ i \left(\Phi_{w,e+z}(h)-T(h)\right)\right\|^2 \geq 1+ \left\|\Phi_{w,e+z}(h)-T(h)\right\|^2,$$ witnessing that $\Phi_{w,e+z}(h)=T(h)$.
Having in mind that $\|\Phi_{w,i h +y}(e+z)\|\leq 1$, and $$\Phi_{w,ih +y}(e+z) = T(e)+ \left(\Phi_{w,e+z}(y)-T(y)\right)+
T(z),$$ Lemma \ref{l FriRu 1.6} implies that $$0=P_1(T(e))\Phi_{w,ih +y}(e+z)=\Phi_{w,e+z}(y)-T(y).$$

$(b)$ Since $T$ is 1-homogeneous, we may assume, without loss of generality, that $\|z\|<1$. Denoting $a= i h + y + z$, it follows that $$T(i h + y + z)=\Phi_{a,w}(e+i h + y + z-e)=\Phi_{a,w}(e+i h+ y+z)-\Phi_{a,w}(e)$$ $$= T(e+i h+ y+z) - \Phi_{a,w}(e)=\hbox{(by $(a)$)}=T(e)+ i T(h)+T(y)+T(z)-\Phi_{a,w}(e).$$

On the other hand, $$T(i h + y + z)=\Phi_{a,e+z}(e+i h + y + z-e)$$ $$=T(e+z)+i \Phi_{a,e+z}(h)+ \Phi_{a,e+z}(y)-\Phi_{a,e+z}(e)= \hbox{(by Proposition \ref{p 2-local triple homomorphisms are OA})}$$ $$= T(e) +T(z)+i \Phi_{a,e+z}(h)+ \Phi_{a,e+z}(y) -\Phi_{a,e+z}(e).$$

We conclude from Lemma \ref{l triple hom with z open unit ball} that $\Phi_{a,e+z}(e) = T(e)$. Therefore $\Phi_{a,e+z}(h)\in F_2 (T(e))_{sa}$, $\Phi_{a,e+z}(y)\in F_1(T(e)),$
$$T(i h + y + z)=i \Phi_{a,e+z}(h)+ \Phi_{a,e+z}(y)+ T(z),$$ and hence $$\Phi_{a,w}(e) = T(e) + i (T(h)-\Phi_{a,e+z}(h)) + (T(y)-\Phi_{a,e+z}(y)).$$ The arguments given in the final part of the proof of $(a)$ show that $T(h)=\Phi_{a,e+z}(h)$ and  $T(y)=\Phi_{a,e+z}(y)$.\smallskip

$(c)$ For each real number $\lambda\in [0,1]$, we denote $w_{\lambda}:=e+ i h+y+\lambda z$.
By the assumptions
$$T(w_{\lambda})=\Phi_{w_{\lambda},e}(w_{\lambda})=T(e)+ i \Phi_{w_{\lambda},e}(h)+\Phi_{w_{\lambda},e}(y)+\lambda \Phi_{w_{\lambda},e}(z),$$
where $\Phi_{w_{\lambda},e}(h)\in F_2 (T(e))_{sa}$, $\Phi_{w_{\lambda},e}(y)\in F_1(T(e)),$ and $\Phi_{w_{\lambda},e}(z)\in F_0(T(e))$.
Applying $(b)$ we deduce that
$$T(w_{\lambda})=\Phi_{w_{\lambda},ih+y+\lambda z}(w_{\lambda})=i T(h) + T(y)+\lambda
    T(z)+\Phi_{w_{\lambda},ih+y+\lambda z}(e).$$
The above identities show that $$\Phi_{w_{\lambda},ih + y+\lambda z}(e)=T(e)+ i (\Phi_{w_{\lambda},e}(h)-T(h))  +(\Phi_{w_{\lambda},e}(y)-T(y))$$ $$+\lambda (\Phi_{w_{\lambda},e}(z)-T(z)).$$ Repeating again the arguments given in the final part of the proof of $(a)$ we obtain $T(h)=\Phi_{w_{\lambda},e}(h)$ and  $T(y)=\Phi_{w_{\lambda},e}(y)$. Therefore,
$$\Phi_{w_{\lambda},ih+y+\lambda z}(e)=T(e)+\lambda (\Phi_{w_{\lambda},e}(z)-T(z)).$$
Since $\Phi_{w_{\lambda},ih+y+\lambda z}(e)$ is a tripotent, we deduce that $$P_0(T(e))\Phi_{w_{\lambda},ih+y+\lambda z}(e)=\lambda
    (\Phi_{w_{\lambda},e}(z)-T(z))=P_0(T(e))(T(w_{\lambda}))-\lambda T(z)$$ is a tripotent.\smallskip

Finally the mapping $f:[0,1]\to \{0,1\}$ given by
$$f(\lambda)=\| P_0(T(e))(T(w_{\lambda}))-\lambda T(z)\|$$ is (norm) continuous and $f(0)=0$.
Then $f(\lambda)=0$ for every $\lambda \in [0,1]$, and hence $\Phi_{w_{1},e}(z)=T(z),$ which gives the desired statement.
\end{proof}

\section{2-local triple homomorphisms on a JBW$^*$-algebra or on a JBW$^*$-triple}

In this section we establish the main results of the paper. Our study on 2-local triple homomorphisms will culminate in a result asserting that every (not necessarily linear) 2-local triple homomorphism from a JBW$^*$-triple into a JB$^*$-triple is linear and a triple homomorphism. In a first step we consider 2-local triple homomorphisms whose domains are JBW$^*$-algebras.

\subsection{2-local triple homomorphisms on a JBW$^*$-algebra}

The aim of this subsection is to study 2-local triple homomorphisms from a JBW$^*$-algebra or from a von Neumann algebra into a JB$^*$-triple. The results in these settings are interesting by themselves but also play a crucial role in the proof of our main result for JBW$^*$-triples.\smallskip

Let $\Phi : \mathcal{J} \to F$ be a triple homomorphism from a unital JB$^*$-algebra into a JB$^*$-triple. Clearly $\Phi (1)$ is a tripotent in $F$ and $F_2 (\Phi (1))$ is a JB$^*$-algebra with unit $\Phi(1)$. Given $a$ in $\mathcal{J}$, the identities $$\left\{\Phi (1),\Phi(1),\Phi(a) \right\} = \Phi \{1,1,a\} = \Phi (a),$$ and $$\Phi(a)^{\sharp_{\Phi(1)}} = \left\{\Phi (1),\Phi(a),\Phi(1) \right\} = \Phi \{1,a,1\} = \Phi (a^*),$$ prove that $\Phi(\mathcal{J})\subseteq F_2 (\Phi(1))$ and $\Phi(\mathcal{J}_{sa})\subseteq F_2 (\Phi(1))_{sa}$. More precisely, $\Phi$ is $F_2 (\Phi(1))$-valued and $\Phi : \mathcal{J} \to F_2 (\Phi(1))$ is a unital Jordan $^*$-homomorphism between unital JB$^*$-algebras. For 2-local triple homomorphisms we have:

\begin{lemma}\label{l basic 2-local triple hom on a JBW$^*$-algebra} Let $T: \mathcal{J} \to F$ be a (not necessarily linear) 2-local triple homomorphism from a unital JB$^*$-algebra into a JB$^*$-triple. Then the following statements hold:\begin{enumerate}[$(a)$] \item $T(\mathcal{J})\subseteq F_2 (T(1))$;
\item $T(\mathcal{J}_{sa})\subseteq F_2 (T(1))_{sa}$;
\item $T(a)$ is positive in $F_2 (T(1))$ whenever $a$ is positive in $\mathcal{J}$.
\end{enumerate}
\end{lemma}

\begin{proof} For each $a\in \mathcal{J},$ and $b\in \mathcal{J}_{sa}$, the comments preceding this lemma assure that $T(a)= \Phi_{1,a} (a) \in F_2 (\Phi_{1,a}(1)) = F_2 (T(1)),$ and $T(b) = \Phi_{1,b} (b)\in F_2 (\Phi_{1,b}(1))_{sa} = F_2 (T(1))_{sa},$  which proves $(a)$ and $(b)$.\smallskip

To prove $(c)$, suppose $a$ is a positive element in $\mathcal{J}.$ Since the triple homomorphism $\Phi_{a,1} : \mathcal{J} \to F_2 (\Phi_{a,1}(1)) = F_2 (T(1))$ is a unital Jordan $^*$-homomorphism between unital JB$^*$-algebras, $T(a)=\Phi_{1,a} (a)$ is positive in $F_2 (\Phi_{a,1}(1)) = F_2 (T(1))$.
\end{proof}

Let $\mathcal{J}$ be a JB$^*$-algebra and let $E$ be a JB$^*$-triple. Following the notation employed in \cite{Aarnes70} and \cite{BuWri96},
a \emph{quasi-linear functional} on $\mathcal{J}$ is a function $\rho: \mathcal{J} \rightarrow \mathbb{C}$ such that \begin{enumerate}[$(i)$]\item $\rho|_{\mathcal{J}_{<h>}}: \mathcal{J}_{<h>} \to \mathbb{C}$ is a linear functional for each $h\in \mathcal{J}_{sa}$, where $\mathcal{J}_{<h>}$ denotes the JB$^*$-subalgebra generated by $h$;
\item $\rho (a + ib) = \rho (a) + i \rho (b)$, for every $a,b\in \mathcal{J}_{sa}$.
\end{enumerate} If we also assume that, for each $h\in \mathcal{J}_{sa}$, $\rho|_{\mathcal{J}_{<h>}}$ is a positive linear functional, we shall say that $\rho$ is \emph{positive quasi-linear functional} on $\mathcal{J}$. A mapping $\rho : E \to \mathbb{C}$ is said to be a \emph{quasi-linear functional} on $E$ if for every $a$ in $E$, the restriction of $\rho$ to the JB$^*$-subtriple, $E_a$, of $E$ generated by $a$ is linear.\smallskip

Let $T : E \to F$ be a (not necessarily linear) 2-local triple homomorphism between JB$^*$-triples. For each $\phi\in F^*$, Lemma \ref{l linearity on single generated subalgebras} assures that $\phi \circ T : E\to \mathbb{C}$ is a quasi-linear functional on $E$ in the triple sense. Our next proposition shows that a stronger property holds for 2-local triple homomorphisms from a JBW$^*$-algebra into a JB$^*$-triple.

\begin{proposition}\label{p quasi-linearity} Let $T: \mathcal{J} \to F$ be a (not necessarily linear) 2-local triple homomorphism from a JBW$^*$-algebra into a JB$^*$-triple. Then $$T(a+ib)= T(a) +i T(b),$$ for every $a,b\in \mathcal{J}_{sa}$. 
\end{proposition}

\begin{proof} It is known that every $a\in \mathcal{J}_{sa}$, can be approximated in norm by a finite (real) linear combination of mutually orthogonal (non-zero) projections in $\mathcal{J}$ (\cite[Proposition 4.2.3]{HancheStor}). Since $T$ is continuous, it is enough to prove that \begin{equation}\label{eq 1 p quasi-states} T(a+ib)= T(a) +i T(b),
 \end{equation}for every $b\in \mathcal{J}_{sa}$ and $\displaystyle a=\sum_{k=1}^{m} \lambda_k p_k$, where $\lambda_k\in \mathbb{R}\backslash\{0\}$ and $p_1,\ldots,p_m$ are mutually orthogonal projections in $\mathcal{J}$. We shall prove \eqref{eq 1 p quasi-states} by induction on $m$.\smallskip

For the case $m=1$, we assume that $a= \lambda p$ for a non-zero projection $p$ and $\lambda\in \mathbb{R}\backslash\{0\}$. The element $b$ writes in the form $b= P_2 (p) (b)+ P_1(p)(b)+P_0(p)(b)$, and since $b=b^*$ and $p$ is a projection, $P_2 (p) (b) \in \mathcal{J}_{sa}$. Applying Lemma \ref{l 2-local and Peirce arithmetics II}$(c)$, we have $$T(a+ib) =  \lambda T\left(p\right)+i T\left( P_2 (p) (b) \right) + i T\left( P_1(p)(b)\right)+ i T\left(P_0(p)(b)\right) $$
$$=T(a) + i T\left( P_2 (p) (b) \right) + i T\left( P_1(p)(b)\right)+ i T\left(P_0(p)(b)\right),$$
by Lemma \ref{l 2-local and Peirce arithmetics II}$(b)$,
$$=T(a) + T\left( i P_2 (p) (b) + i  P_1(p)(b)+ i P_0(p)(b)\right) =T(a+ib).$$

Suppose, by the induction hypothesis, that \eqref{eq 1 p quasi-states} is true for every algebraic element $\displaystyle \sum_{k=1}^{m_1} \mu_k q_k$, with $m_1\leq m$, $\lambda_k\in \mathbb{R}\backslash\{0\}$ and $q_1,\ldots,q_{m_1}$ mutually orthogonal projections in $\mathcal{J}$. Let us take an algebraic element of the form $\displaystyle a=\sum_{k=1}^{m+1} \lambda_k p_k$. Let us write $b= P_2 (p_1) (b)+ P_1(p_1)(b)+P_0(p_1)(b)$, and since $b=b^*$ and $p_1$ is a projection, $P_2 (p_1) (b), P_0 (p_1) (b) \in \mathcal{J}_{sa}$. Let us observe that $$T(a+ib) = T\left(\lambda_1 p_1 + i P_2 (p_1)(b) + i P_1 (p_1)(b) + \sum_{k=2}^{m+1} \lambda_k p_k + i P_0 (p_1)(b) \right),$$ where $P_2 (p_1) (b)\in \mathcal{J}_2 (p_1)_{sa}$, $i P_1 (p_1)(b)\in \mathcal{J}_1 (p_1)$, and  $\sum_{k=2}^{m+1} \lambda_k p_k + i P_0 (p_1)(b)$ lies in $\mathcal{J}_0 (p_1)$. Lemma \ref{l 2-local and Peirce arithmetics II}$(c)$ implies that
$$T(a+ib) = \lambda_1 T\left( p_1\right) + i T\left(P_2 (p_1)(b)\right) + T\left(i P_1 (p_1)(b) \right)$$ $$+ T\left(\sum_{k=2}^{m+1} \lambda_k p_k + i P_0 (p_1)(b) \right) = \hbox{(by the induction hypothesis)}= $$ $$= \lambda_1 T\left( p_1\right) + i T\left(P_2 (p_1)(b)\right) + T\left(i P_1 (p_1)(b) \right)+ T\left(\sum_{k=2}^{m+1} \lambda_k p_k\right) + T\left(i P_0 (p_1)(b) \right)$$ $$=  \hbox{(by Proposition \ref{p 2-local triple homomorphisms are OA})}=  T\left(\lambda_1 p_1 + \sum_{k=2}^{m+1} \lambda_k p_k \right) + i T\left(P_2 (p_1)(b)\right) $$ $$+ T\left(i P_1 (p_1)(b) \right)+ T\left(i P_0 (p_1)(b) \right)  = \hbox{(by Lemma \ref{l 2-local and Peirce arithmetics II}$(b)$ with $e=p_1$)}$$ $$= T(a) + i T\left(P_2 (p_1)(b)+  P_1 (p_1)(b) + P_0 (p_1)(b) \right)  =T(a) + i T(b).$$
\end{proof}

We can establish now a generalization of \cite[Theorem 4.1]{BurFerGarPe2014preprint} for 2-local triple homomorphisms.

\begin{theorem}\label{t 2-local hom on JBW$^*$-algebras not containing M2 nor S2}
Let $\mathcal{J}$ be a JBW$^*$-algebra with no Type $I_2$ direct
summand and let $F$ be a JB$^*$-triple. Suppose $T : \mathcal{J} \to
F$ is a (not necessarily linear) 2-local triple homomorphism.
Then $T$ is linear and a (continuous) triple homomorphism. More concretely, $T:\mathcal{J} \to
F_2 (T(1))$ is a linear unital Jordan $^*$-homomorphism between JB$^*$-algebras.
\end{theorem}

\begin{proof} We know that $T(1)$ is a tripotent in $F$ (cf. Lemma \ref{l basic properties}). By Lemma \ref{l basic 2-local triple hom on a JBW$^*$-algebra} $T(\mathcal{J})\subseteq F_2 (T(1))$ and $T(\mathcal{J}_{sa})\subseteq F_2 (T(1))_{sa}$.
\smallskip

Furthermore, given a projection $p$ in $\mathcal{J}$, since $p\leq 1$, Lemma \ref{l basic properties}$(b)$ and $(d)$ assure that $T(p)$ is a tripotent in $F_2 (T(1))$ with $T(p)\leq T(1),$ which implies that $T(p)$ is a projection in $F_2 (T(1)).$\smallskip

Fix an arbitrary norm-one positive functional $\varphi$ in $F_2 (T(1))^{*}$. Let $\mathcal{P} (\mathcal{J})$ denote the lattice of projections in $\mathcal{J}$. The mapping $$\mu_{\varphi} :\mathcal{P} (\mathcal{J}) \to \mathbb{R} $$ $$\mu_{\varphi} (p) := \varphi (T(p)),$$ is a finitely additive quantum measure on $\mathcal{P} (\mathcal{J})$ in the terminology employed in \cite{BuWri89}, i.e. $\mu_{\varphi} (1)=1$ and $\mu_{\varphi} (p_1+\ldots+p_m ) = \mu_{\varphi} (p_1)+\ldots+ \mu_{\varphi} (p_m )$, whenever $p_1,\ldots, p_m$ are mutually orthogonal projections in $\mathcal{J}$ (this statement follows from Lemma \ref{l linearity on orthogonal tripotents}$(a)$ and the fact that $T(1)$ is the unit element in $F_2 (T(1))$). By the Bunce-Wright-Mackey-Gleason theorem \cite[Theorem 2.1]{BuWri89}, there exists a positive linear functional $\phi_{\varphi}\in \mathcal{J}_{sa}^*$ such that $$ \varphi (T(p)) = \mu_{\varphi} (p) = \phi_{\varphi} (p),$$ for every $p\in \mathcal{P} (\mathcal{J}).$ It follows from Lemma \ref{l linearity on orthogonal tripotents}$(b)$, the continuity of $T$, $\varphi$, and $\phi_{\varphi}$, and the norm density of algebraic elements in $\mathcal{J}_{sa}$ that $$ \varphi (T(a)) =  \phi_{\varphi} (a),$$ for every $a\in \mathcal{J}_{sa}.$ Therefore, $$\varphi \left(T(a+b)\right) = \phi_{\varphi} (a+ b) = \phi_{\varphi} (a) + \phi_{\varphi} (b)$$ $$= \varphi \left(T(a)\right) + \varphi \left(T(b) \right) = \varphi (T(a)+T(b)),$$ for every $a,b\in \mathcal{J}_{sa}.$ Since the positive norm-one functionals in $F_2 (T(1))^{*}$ separate the points of $F_2 (T(1))_{sa}$ (cf. \cite[Lemma 3.6.8]{HancheStor}), we deduce that $T(a+b) = T(a) + T(b)$, for every $a,b\in \mathcal{J}_{sa}$, that is, the restricted mapping $T|_{\mathcal{J}_{sa}} : \mathcal{J}_{sa} \to F_2(T(1))_{sa}\subseteq F$ is linear.\smallskip

Finally, Proposition \ref{p quasi-linearity} shows that $T(a+ib) = T(a) + i T(b)$, for every $a,b\in \mathcal{J}_{sa}$, which gives the desired statement.
\end{proof}

When in the proof of \cite[Corollary 4.4]{BurFerGarPe2014preprint} (respectively, \cite[Corollary 2.11]{BurFerGarPe2014preprint}),
\cite[Proposition 4.2 and Corollary 4.3]{BurFerGarPe2014preprint} (respectively, \cite[Proposition 2.7 and Corollary 2.10]{BurFerGarPe2014preprint})
are replaced with Corollary \ref{c 2-local triple hom on M2 and S2} and Corollary \ref{c stability for ell infty sums}, respectively, and having in mind Proposition \ref{p quasi-linearity}, the arguments in those results remain valid to prove:

\begin{corollary}\label{c type I2}
Every (not necessarily linear) 2-local triple homomorphism from a
Type $I_2$ JBW$^*$-algebra into a JB$^*$-triple is linear and a triple homomor-phism.$\hfill\Box$
\end{corollary}

The first main result of this note is a consequence of Theorem \ref{t 2-local hom on JBW$^*$-algebras not containing M2 nor S2},
Corollary \ref{c type I2} and Corollary \ref{c stability for ell infty sums}.

\begin{theorem}\label{t 2-local triple hom on JBW$^*$ algebras}
Every (not necessarily linear) 2-local triple homomorphism from a
JBW$^*$-algebra into a JB$^*$-triple is linear and a triple
homomorphism.$\hfill\Box$
\end{theorem}

The next corollary is interesting by itself.

\begin{corollary}\label{c 2-local triple hom on von Neumann algebras}
Every (not necessarily linear) 2-local triple homomorphism from a
von Neumann algebra into a JB$^*$-triple is linear and a triple
homomorphism.
\end{corollary}

Since every $^*$-homomorphism between C$^*$-algebras (respectively, every Jordan $^*$-homomorphism between JB$^*$-algebras) is a triple homomorphism, Theorems 2.12 and 4.5 and Corollary 4.6 in \cite{BurFerGarPe2014preprint} are direct consequences of the previous Theorem \ref{t 2-local triple hom on JBW$^*$ algebras} and Corollary \ref{c 2-local triple hom on von Neumann algebras}.\smallskip

\subsection{2-local triple homomorphisms on a JBW$^*$-triple}

The rest of the note is devoted to prove the second main result of the paper, in which we shall show that every (not necessarily linear) 2-local triple homomorphism from a JBW$^*$-triple into a JB$^*$-triple is linear and a triple homomorphism. The first step toward our goal is the following corollary.

\begin{corollary}\label{c 2-local triple hom Peirce-2}
Let $T: E \to F$ be a (not necessarily linear) 2-local triple homomorphism from a JBW$^*$-triple into a JB$^*$-triple.
Then, for each tripotent $e$ in $E$, $T|_{E_2(e)} : E_2 (e) \to F$ is linear and a triple homomorphism.
\end{corollary}

\begin{proof}
Clear from Theorem \ref{t 2-local triple hom on JBW$^*$ algebras}.
\end{proof}

We are now in a position to establish the goal of this section.

\begin{theorem}\label{t 2-local triple hom on JBW$^*$ triples}
Every (not necessarily linear) 2-local triple homomorphism from a
JBW$^*$-triple into a JB$^*$-triple is linear and a triple
homomorphism.
\end{theorem}

\begin{proof} Let $x, y$ be two (arbitrary) elements in
$E$. Find a complete tripotent $e$ in $E$ such that $x\in E_2(e)$ (the existence of such a tripotent in guaranteed by \cite[Lemma 3.12$(1)$]{Horn}). If we write $y=P_2(e)y+P_1(e)y$, then Proposition \ref{p 2-local in Peirce-2 Peirce-1 arithmetics}
and Corollary \ref{c 2-local triple hom Peirce-2} prove that
$$T(x+y)=T(x+P_2(e)y+P_1(e)y)=T(x+P_2(e)y)+T(P_1(e))$$ $$=T(x)+T(P_2(e)y)+T(P_1(e)y).$$
A new application of Proposition \ref{p 2-local in Peirce-2 Peirce-1
arithmetics} shows that $T(P_2(e)y)+T(P_1(e)y)=T(P_2(e)y+P_1(e)y),$ and hence $T(x+y)=T(x)+T(y).$
\end{proof}

\end{document}